\documentclass[12pt]{amsart}
\usepackage{amsmath,amssymb,amsbsy,amsfonts,amsthm,latexsym, amsopn,amstext,amsxtra,euscript,amscd,mathrsfs,color,bm, cite,multirow,tabularx,dirtytalk}
\usepackage{todonotes}   
\usepackage{url}
\usepackage[colorlinks,linkcolor=blue,anchorcolor=blue,citecolor=blue]{hyperref}
\usepackage{color}
\usepackage{comment}

\begin{document}
\newtheorem{problem}{Problem}
\newtheorem{theorem}{Theorem}
\newtheorem{lemma}[theorem]{Lemma}
\newtheorem{crit}[theorem]{Criterion}
\newtheorem{claim}[theorem]{Claim}
\newtheorem{cor}[theorem]{Corollary}
\newtheorem{prop}[theorem]{Proposition}
\newtheorem{definition}{Definition}
\newtheorem{question}[theorem]{Question}
\newtheorem{rem}[theorem]{Remark}
\newtheorem{Note}[theorem]{Notation}

\def\cA{{\mathcal A}}
\def\cB{{\mathcal B}}
\def\cC{{\mathcal C}}
\def\cD{{\mathcal D}}
\def\cE{{\mathcal E}}
\def\cF{{\mathcal F}}
\def\cG{{\mathcal G}}
\def\cH{{\mathcal H}}
\def\cI{{\mathcal I}}
\def\cJ{{\mathcal J}}
\def\cK{{\mathcal K}}
\def\cL{{\mathcal L}}
\def\cM{{\mathcal M}}
\def\cN{{\mathcal N}}
\def\cO{{\mathcal O}}
\def\cP{{\mathcal P}}
\def\cQ{{\mathcal Q}}
\def\cR{{\mathcal R}}
\def\cS{{\mathcal S}}
\def\cT{{\mathcal T}}
\def\cU{{\mathcal U}}
\def\cV{{\mathcal V}}
\def\cW{{\mathcal W}}
\def\cX{{\mathcal X}}
\def\cY{{\mathcal Y}}
\def\cZ{{\mathcal Z}}

\def\A{{\mathbb A}}
\def\B{{\mathbb B}}
\def\C{{\mathbb C}}
\def\D{{\mathbb D}}
\def\E{{\mathbb E}}
\def\F{{\mathbb F}}
\def\G{{\mathbb G}}
\def\I{{\mathbb I}}
\def\J{{\mathbb J}}
\def\K{{\mathbb K}}
\def\L{{\mathbb L}}
\def\M{{\mathbb M}}
\def\N{{\mathbb N}}
\def\O{{\mathbb O}}
\def\P{{\mathbb P}}
\def\Q{{\mathbb Q}}
\def\R{{\mathbb R}}
\def\S{{\mathbb S}}
\def\T{{\mathbb T}}
\def\U{{\mathbb U}}
\def\V{{\mathbb V}}
\def\W{{\mathbb W}}
\def\X{{\mathbb X}}
\def\Y{{\mathbb Y}}
\def\Z{{\mathbb Z}}
\def\OK{{\mathcal{O}_\K}}

\def\ep{{\mathbf{e}}_p}
\def\eq{{\mathbf{e}}_q}
\def\cal#1{\mathcal{#1}}

\def\scr{\scriptstyle}
\def\\{\cr}
\def\({\left(}
\def\){\right)}
\def\[{\left[}
\def\]{\right]}
\def\<{\langle}
\def\>{\rangle}
\def\fl#1{\left\lfloor#1\right\rfloor}
\def\rf#1{\left\lceil#1\right\rceil}
\def\le{\leqslant}
\def\ge{\geqslant}
\def\eps{\varepsilon}
\def\mand{\qquad\mbox{and}\qquad}

\def\sssum{\mathop{\sum\ \sum\ \sum}}
\def\ssum{\mathop{\sum\, \sum}}
\def\ssumw{\mathop{\sum\qquad \sum}}

\def\vec#1{\mathbf{#1}}
\def\inv#1{\overline{#1}}
\def\num#1{\mathrm{num}(#1)}
\def\dist{\mathrm{dist}}

\def\fA{{\mathfrak A}}
\def\fB{{\mathfrak B}}
\def\fC{{\mathfrak C}}
\def\fU{{\mathfrak U}}
\def\fV{{\mathfrak V}}

\newcommand{\bflambda}{{\boldsymbol{\lambda}}}
\newcommand{\bfxi}{{\boldsymbol{\xi}}}
\newcommand{\bfrho}{{\boldsymbol{\rho}}}
\newcommand{\bfnu}{{\boldsymbol{\nu}}}
\newcommand{\norm}[1]{\left\lVert#1\right\rVert}

\def\GL{\mathrm{GL}}
\def\SL{\mathrm{SL}}
\def\Log{\mathrm{Log}}

\def\Hba{\overline{\cH}_{a,m}}
\def\Hta{\widetilde{\cH}_{a,m}}
\def\Hb1{\overline{\cH}_{m}}
\def\Ht1{\widetilde{\cH}_{m}}

\def\flp#1{{\left\langle#1\right\rangle}_p}
\def\flm#1{{\left\langle#1\right\rangle}_m}

\def\Zm{\Z/m\Z}

\def\Err{{\mathbf{E}}}
\def\O{\mathcal{O}}

\def\cc#1{\textcolor{red}{#1}}
\newcommand{\commT}[2][]{\todo[#1,color=green!60]{Tim: #2}}
\newcommand{\commB}[2][]{\todo[#1,color=red!60]{Bryce: #2}}

\newcommand{\comm}[1]{\marginpar{%
\vskip-\baselineskip 
\raggedright\footnotesize
\itshape\hrule\smallskip#1\par\smallskip\hrule}}
\newcolumntype{L}{>{\raggedright\arraybackslash}X}

\def\xxx{\vskip5pt\hrule\vskip5pt}

\def\dmod#1{\,\left(\textnormal{mod }{#1}\right)}

\title[The covering radius of the logarithmic lattice of $\Q(\zeta_n)$]{An improved upper bound on the covering radius of the logarithmic lattice of $\Q(\zeta_n)$}

\author[J. Punch]{James Punch}
\address{School of Science, The University of New South Wales, Canberra, Australia}
\email{j.punch@unsw.edu.au}
 
\date{\today}
\pagenumbering{arabic}

\begin{abstract}
Let $\R^m$ be endowed with the Euclidean metric. The covering radius of a lattice $\Lambda \subset \R^m$ is the least distance $r$ such that, given any point of $\R^m$, the distance from that point to $\Lambda$ is not more than $r$. Lattices can occur via the unit group of the ring of integers in an algebraic number field $\K$, by applying a logarithmic embedding $\K^*\rightarrow \R^m$. In this paper, we examine those lattices which arise from the cyclotomic number field $\Q(\zeta_n)$, for a given positive integer $n\geq5$ such that $n\not \equiv 2\pmod{4}$. We then provide improvements to a result of de Araujo in \cite{deAraujo2024}, and conclude with an upper bound on the covering radius for this lattice in terms of $n$ and the number of its distinct prime factors. In particular, we improve \cite[Lemma 2]{deAraujo2024}, and show that, asymptotically, it can be improved no further.
\end{abstract}
\maketitle
\section{Introduction}
Suppose that $n$ is a positive integer not less than 5, and that $n\not\equiv 2\pmod{4}$. Let $\zeta_n$ be a primitive $n^\text{th}$ root of unity, and let $\K=\Q(\zeta_n)$. We denote by $\OK$ the ring of integers in $\K$, and its unit group by $\OK^*$. As in \cite[Theorem 38]{Marcus1977}, we let $\sigma_1,\ldots,\sigma_r$ be the real embeddings of $\K\rightarrow\R$, and $\sigma_{r+1},\ldots,\sigma_{r+s}$ be representatives of the pairs of complex embeddings $\K\rightarrow\C$. Then the map $\Log:\K^*\rightarrow\R^{r+s}$ is defined by $x\mapsto (\log\lvert \sigma_1(x) \rvert,\ldots,\log\lvert\sigma_r(x)\rvert,\log(\lvert\sigma_{r+1}(x)\rvert^2),\ldots,\log(\lvert\sigma_{r+s}(x)\rvert^2)$, and is called the \textit{logarithmic embedding}. Typically, when proving Dirichlet's unit theorem, one shows that $\Log(\OK^*)$ is a lattice of rank $r+s-1$. For cyclotomic fields, the minimal polynomial of $\zeta_n$ is the $n^\text{th}$ cyclotomic polynomial, whose degree is $\varphi(n)$. There can be no embeddings $\K\rightarrow\R$, so $r$ is $0$ and $s$ is $\varphi(n)/2$. Then the lattice $\Log(\OK^*)$, which we call the \textit{logarithmic lattice}, has rank $\frac{\varphi(n)}{2}-1$.

When $\varphi(n)=2$ (equivalently, $n\in\{3,4\}$), the rank is $0$. That is, $\O_\K^*$ will have finitely many units, all of which map to the zero vector under $\Log$. It is absurd to ask about the covering radius of a rank-zero lattice, so we restrict to $n\geq 5$.

We now describe some properties of lattices, as in the introduction of \cite{deAraujo2024}. The \textit{$i^\text{th}$ successive minimum} of a lattice $\Lambda$, denoted $\lambda_i(\Lambda)$, is the value $$\inf_{\substack{v_1,\ldots,v_i\in\Lambda\\\{v_1,\ldots,v_i\}\text{ is linearly independent}}}\max_{1\leq j\leq i}\norm{v_j}.$$ If $R=\lambda_i(\Lambda)$, this definition means that $\Lambda$ contains $i$ linearly independent vectors of length less than or equal to $R$, and $R$ is the least value for which this holds. The special case $\lambda_1(\Lambda)$ is the length of $\Lambda$'s shortest nonzero vector. The \textit{covering radius} of $\Lambda$, denoted $\mu(\Lambda)$, is the quantity $\max_{x\in\text{span}(\Lambda)}\min_{v\in \Lambda}\norm{x-v}.$ That is, given a fixed $x$ in the span of $\Lambda$, we find the shortest distance to a point of $\Lambda$. Then, select the largest such distance over all possible $x$. It is called a covering radius because, if $r=\mu(\Lambda)$, the set of closed Euclidean balls of radius $r$ centred at points of $\Lambda$ will cover the span of $\Lambda$, with $r$ the least value for which this holds.

The strategy of \cite{deAraujo2024}, which we also use, is to select a finite-index subgroup $C$ of $\OK^*$, which necessarily gives a full-rank sublattice $\Log\, C$ of $\Log \,\OK^*$. We have $\mu(\Log \, \OK^*)\leq\mu(\Log\, C)$. Lemma \ref{lemma:coveringRadius} provides an upper bound on $\mu(\Log \, C)$ in terms of $\lambda_r(\Log \, C)$. Lemma \ref{lemma:Ramachandra}) gives a basis for $\Log\, C$ and the lengths of the basis vectors are bounded above by Lemma \ref{lemma:theirLemma4}. But $\lambda_r(\Log \, C)$ is no more than the length of the longest basis vector. Combining these results, we get an upper bound on the covering radius of $\Log \,\OK^*$.

We now briefly discuss other work on the covering radii of logarithmic lattices. In \cite{CDPR2016}, Cramer, Ducas, Peikert and Regev assess an attack on a particular cryptographic protocol involving such lattices; part of the attack reduces to the \textit{Closest Vector Problem}. This problem asks, given a lattice $\Lambda$ and a point $P$ in its span, to find the vector $l\in\Lambda$ that is closest to $P$ (with respect to a chosen metric). The distance from $P$ to its \say{closest vector} cannot exceed the covering radius of the lattice, by definition. If we are searching for the closest vector, the covering radius indicates the size of the search space, or the hardness of the problem.

In \cite{CDPR2016}, they use the metric induced by the $\ell_\infty$ norm, and the relevant result is that the covering radius of the logarithmic lattice of $\Q(\zeta_m)$ is of order $O(\sqrt{m}\log{m})$.
In this paper and in \cite{deAraujo2024}, the metric induced by the $\ell_2$ norm is used instead.

Other authors have considered this problem for different classes of number fields. In \cite{ATPS2021}, the authors derive upper bounds for the covering radius of the logarithmic lattice of $\K$, where $\K$ is a biquadratic number field.

We now outline our improvements to the technique in \cite{deAraujo2024}. \cite[Lemma 2]{deAraujo2024} gives the inequality $$\sum_{k=1}^{\lfloor m/2 \rfloor}\log^2\left(2\sin(\pi k/m)\right)<\frac{5}{4}m,$$ for any integer $m>1$. We show that, as $m\rightarrow\infty$, this sum is asymptotic to $\frac{\pi^2}{24}m$. This result immediately implies an improved version of \cite[Lemma 3]{deAraujo2024}. We then improve the bound on the sums in \cite[Lemma 4]{deAraujo2024}, by observing that only one term in each sum can be large, and all other terms must be comparatively small. Finally, the proof of \cite[Theorem 1]{deAraujo2024} involves the product of $\frac{1}{2}\sqrt{\varphi(n)/2-1}$ and $(2^s-1)$, where $s$ is the number of distinct prime factors of $n$. The first term is largest if $s$ is small, and the second term is largest if $s$ is large. These situations cannot occur simultaneously. We can therefore improve Theorem \ref{thm:betterTheorem1} by bounding $\varphi(n)$ in terms of $s$, which we do in Lemma \ref{lemma:boundOnPhi}.

We now state the main theorem of this paper, and the main theorem of \cite{deAraujo2024}. A comparison of the two bounds for selected values of $n$ is shown in Table \ref{tab:theirTable1}.

\begin{theorem} \textbf{(de Araujo, 2024)}\label{thm:oldTheorem1}
    Let $\K=\Q(\zeta_n)$, where $n\geq 5$ and $n\not\equiv 2\pmod{4}$. The covering radius of the logarithmic lattice $\Log(\mathcal{O}_K^*)$ with respect to the $\ell^2$ norm is less than or equal to \begin{equation}\label{eq:deAraujoUpperBound} n(2^s-1)\sqrt{6},\end{equation} where $s$ is the number of distinct prime factors of $n$.
\end{theorem}
\begin{rem}
    In \cite{deAraujo2024}, the upper bound is instead written as $n(2^s-1)\sqrt{3}$. However, there is an error in Lemma 3 of that paper, within Equation 16. The upper bound in that lemma changes from $\sqrt{3n}$ to $\sqrt{5n+1}$ once the error is corrected, which can be replaced with $\sqrt{6n}$ because $n\geq 1$.
    Also, it is easily observed that the corrected argument in \cite{deAraujo2024} implies that equality cannot hold in (\ref{eq:deAraujoUpperBound}), so we may omit the words \say{or equal to}.
\end{rem}
\begin{theorem} \label{thm:betterTheorem1}
    Let $\K=\Q(\zeta_n)$, where $n\geq 5$ and $n\not\equiv 2\pmod{4}$. The covering radius of the logarithmic lattice $\Log(\mathcal{O}_\K^*)$ with respect to the $\ell^2$ norm is less than $$\sqrt{\frac{n}{2}}\left(1-\frac{1}{\sqrt[s]{n}}\right)^{s/2}\left[(2^s-2)\sqrt{\frac{5\pi^2}{48}n+\frac{15}{2}\log^2{2}}+\sqrt{\frac{\pi^2}{6}n+6\log^2{2}}\right],$$ where $s$ is the number of distinct prime factors of $n$.
\end{theorem}
The following bound, though crude, is perhaps more instructive. It follows immediately from Theorem \ref{thm:betterTheorem1}.
\begin{cor}\label{cor:crudeTheorem1}
    Let $\K=\Q(\zeta_n)$, where $n\geq 5$ and $n\not\equiv 2\pmod{4}$. The covering radius of the logarithmic lattice $\Log(\mathcal{O}_\K^*)$ with respect to the $\ell^2$ norm, $\mu(\Log(\mathcal{O}_\K))$, is less than $$n(2^s-1)\left(1-\frac{1}{\sqrt[s]{n}}\right)^{s/2}\sqrt{1.303},$$ where $s$ is the number of distinct prime factors of $n$.
\end{cor}
\begin{rem}
    Consequently, the upper bound in \cite[Theorem 1]{deAraujo2024} is correct as originally written, notwithstanding the error in \cite[Lemma 3]{deAraujo2024}.
\end{rem}

\begin{table}
    \centering
    \begin{tabular}{|c|c|c|c|c|}
    \hline
        $n$ & $s$ & $\mu(\Lambda_n)$ & Upper bound using \cite[Theorem 1]{deAraujo2024} & Upper bound using Theorem \ref{thm:betterTheorem1}\\
    \hline
        7 & 1 & 1.4 & 12.12 & 6.58\\
        9 & 1 & 2.3 & 15.59 & 8.42\\
        11 & 1 & 5.4 & 19.05 & 10.25\\
        15 & 2 & 2.7 & 77.94 & 28.39\\
        16 & 1 & 6.3 & 27.71 & 14.80\\
    \hline
    \end{tabular}
    \caption{A comparison of the actual covering radius with the upper bounds provided in this work and in \cite{deAraujo2024}, for selected values of $n$. All values are approximate.}
    \label{tab:theirTable1}
\end{table}

\section{Proof of Theorem \ref{thm:betterTheorem1}}
Let $\L$ be the maximal real subfield of $\K$, which is $\Q(\zeta_n+\zeta_n^{-1})$. Then $\O_\L^*$ is a finite-index subgroup of $\OK^*$. The following lemma, due to Ramachandra, is \cite[Lemma 1]{deAraujo2024}. See \cite[Theorem 8.3]{Washington1997} for a proof.
\begin{lemma} \label{lemma:Ramachandra}
    If $n\not\equiv 2\pmod{4}$, then let its prime factorisation be $n=\prod_{i=1}^s {p_i}^{e_i}$. Let $\Gamma$ be the set of strict subsets of $\{1,2,\ldots,s\}$, and for any $I\in\Gamma$ let $n_I=\prod_{i\in I}p_i^{e_i}$. For any integer $a$ coprime to $n$ such that $1<a<n/2$, let $$\xi_a=\prod_{I\in\Gamma} \zeta_n^{n_I(1-a)/2}\left(\frac{1-\zeta_n^{an_I}}{1-\zeta_n^{n_I}}\right).$$ Then $S=\{\xi_a:a\in\Z, 1<a<n/2, \gcd{(a,n)}=1\}$ is a set of multiplicatively independent units for $\L$. Additionally, the group generated by $-1$ and $S$, which we will call $C$, is a finite-index subgroup of $\O_\L^*$.
\end{lemma}
Then $C$ is also a finite-index subgroup of $\OK^*$. The next three lemmas will bound $\norm{\Log{(\xi_a)}}$ from above.
\begin{lemma} \label{lemma:sumAsymptotic}
    For any integer $m>1$, we have 
    $$\frac{\pi^2}{24}m-2.45-\log^2{m}<\sum_{k=1}^{\lfloor m/2 \rfloor}\log^2\left(2\sin(\pi k/m)\right)<\frac{\pi^2}{24}m+\log^2{2}.$$ In particular, $\sum_{k=1}^{\lfloor m/2 \rfloor}\log^2\left(2\sin(\pi k/m)\right)\sim \frac{\pi^2}{24}m$ as $m\rightarrow\infty$.
\end{lemma}
\begin{proof}
    We consider this as a Riemann sum, and approximate it using an integral. Let $f(x)=\log^2(2\sin \pi x)$. The arguments given to this function in the sum are greater than 0, and do not exceed $1/2$. It is easy to see that $f$ is strictly decreasing on the interval $(0,1/6)$ and strictly increasing on the interval $(1/6,1/2)$. We therefore split the sum into two parts and factor out $m$:
    $$\sum_{k=1}^{\lfloor m/2 \rfloor}f(k/m)=m\sum_{k=1}^{\lfloor m/6 \rfloor}\frac{1}{m}f(k/m)+m\sum_{k=\lfloor m/6 \rfloor + 1}^{\lfloor m/2 \rfloor}\frac{1}{m}f(k/m)=S_1+S_2,$$ say. In $S_1$, we may shift the Riemann sum left by one unit so that it sits beneath $f(x)$. Therefore, we have $$S_1< m\int_{0}^{m/6}\frac{1}{m}f(k/m) \,dk= m\int_{0}^{1/6}f(u), $$ by letting $k=mu$.
    Except for the last term in its sum, $S_2$ already lies beneath $f(x)$. We have $$
        S_2=f\left(\frac{\lfloor m/2 \rfloor}{m}\right)+m\sum_{k=\lfloor m/6 \rfloor +1}^{\lfloor m/2 \rfloor - 1}\frac{1}{m}f(k/m)< f(1/2)+ m\int_{\lfloor m/6 \rfloor +1}^{\lfloor m/2 \rfloor}\frac{1}{m}f(k/m) \, dk,$$ using the fact that $\frac{\lfloor m/2 \rfloor}{m}$ is at least $1/6$ whenever $m\geq 2$.
    Substitute $k=mu$ again, and observe that $\frac{\lfloor m/6 \rfloor+1}{m}>\frac{1}{6}$ and $\frac{\lfloor m/2 \rfloor}{m}<\frac{1}{2}$. Since $f(x)$ is always nonnegative, we have $$
        S_2< \log^2{2}+ m\int_{1/6}^{1/2} f(u) \, du \implies S_1+S_2<m\int_{0}^{1/2}f(u) \, du+\log^2{2}.$$ The second inequality of the lemma follows because $\int_{0}^{1/2}f(u) \, du=\frac{\pi^2}{24}.$
    Now, we establish a lower bound. For this, it will be helpful to have a upper bound on the quantity $m\int_{0}^{1/m}f(u)\, du.$ One can show that $f(u)\leq\log^2(6u)$ if $u\in(0,1/6]$, and that $f(u)\leq \log^2(2\pi u)$ if $u\in[1/6,1/2]$.
    If $m\geq 6$, then \begin{align*}m\int_{0}^{1/m}f(u)\, du&\leq  m\int_{0}^{1/m}\log^2(6u)\, du=\log^2(6/m)-2\log(6/m)+2 \\
    &=\log^2 6-2(\log6)(\log{m})+\log^2(m)-2\log{6}+2\log{m}+2\\
    &\leq \log^2{m}+\log^2{6}-2\log{6}+2<\log^2{m}+1.627.
    \end{align*}
    If $m< 6$, then \begin{align*}m\int_{0}^{1/m}f(u)\, du&\leq m\int_{0}^{1/6}\log^2{6u}\, du+m\int_{1/6}^{1/m}\log^2{2\pi u}\, du \\
    &=\frac{m}{3}+m\int_{1/6}^{1/2}\log^2{2\pi u}\, du<m\left(\frac{1}{3}+0.19216\right)<2.628.
    \end{align*}
    The bound $m\int_{0}^{1/m}f(u)\, du\leq 2.2+\log^2{m}$ then holds for any integer $m>1$.
    
    For our lower bound, we shift $S_2$ left by one unit, so that it lies above $f(x)$. If $m$ is even, we have $\lfloor{m/2}\rfloor=m/2$, so $S_2$ (in its new position) extends to $m/2$ on the right. Therefore \begin{align*}
        S_1+S_2\geq m\int_{1}^{m/2}\frac{1}{m}f(k/m)\, dk &= m\int_{0}^{1/2}f(u)\, du - m\int_0^{1/m}f(u)\, du \\
        &> \frac{\pi^2}{24}m-2.2-\log^2{m}.
    \end{align*}
    If $m$ is odd, however, we have $\lfloor{m/2}\rfloor=(m-1)/2$, therefore \begin{align*}
        S_1+S_2&\geq m\int_{1}^{(m-1)/2}\frac{1}{m}f(k/m)\, dk =m\int_{1/m}^{(1-1/m)/2}f(u)\, du\\
        &= m\int_{0}^{1/2}f(u)\, du- m\int_{0}^{1/m}f(u)\, du -m\int_{(1-1/m)/2}^{1/2}f(u)\, du \\
        &> \frac{\pi^2}{24}m-2.2-\log^2{m}-m\int_{(1-1/m)/2}^{1/2}f(u)\, du.
    \end{align*}
    We now need an upper bound on the last term. This is easy; we have
    \begin{align*}
        m\int_{(1-1/m)/2}^{1/2}f(u)\, du &\leq  m\int_{(1-1/m)/2}^{1/2}f(1/2)\, du
        \leq m\cdot\frac{1}{2m}\log^2{2}=\frac{\log^2{2}}{2},
    \end{align*} provided that $\frac{1-1/m}{2}\geq \frac{1}{6}$, which happens if $m\geq 3/2$.
    So, for odd $m$, we find that $$S_1+S_2\geq \frac{\pi^2}{24}m-2.2-\log^2{m}-\frac{1}{2}\log^2{2}> \frac{\pi^2}{24}m-2.45-\log^2{m}.$$ Combining the results for odd and even $m$, we get the first inequality.
\end{proof}
\begin{rem}
    One could improve the upper and lower bounds on the sum in Lemma \ref{lemma:sumAsymptotic} using Euler-Maclaurin summation. However, we see that the main term is $\frac{\pi^2}{24}m$ with a small error term of order $O(\log^2{m})$, so only small improvements are possible.
\end{rem}
\begin{lemma} \label{lemma:theirLemma3}
    Let $a$ and $n>1$ be integers such that $1\leq a<n/2$ and $\gcd{(a,n)}=1$. Consider $I\in \Gamma=\{I \subsetneq [s]\}$ and $n_I=\prod_{i\in I} {p_i}^{e_i}$, as in Lemma \ref{lemma:Ramachandra}. Then
    \begin{equation}\label{eq:LogBound}\norm{\Log\left(\zeta_n^{-an_I/2}-\zeta_n^{an_I/2}\right)}<\sqrt{\frac{\pi^2}{24}n+\frac{9}{2}\log^2{2}+\frac{3\log^2{2}}{2}n_I+\frac{3\pi^2}{24}\frac{n}{n_I}}.\end{equation}
\end{lemma}
\begin{proof}
    Our proof is almost identical to that of \cite[Lemma 3]{deAraujo2024}, but we substitute our improved bound from Lemma \ref{lemma:sumAsymptotic}, and we leave the bound in terms of $n_I$. We have, from \cite[Equation 11]{deAraujo2024}, that:
    \begin{equation}\norm{\Log\left(\zeta_n^{-an_I/2}-\zeta_n^{an_I/2}\right)}^2\leq \lceil n_I/4 \rceil \sum_{k=1}^{2m_I}\log^2\left\lvert2\sin(\pi k/m_I) \right\rvert.\label{eq:deAraujo11}\end{equation} This sum can then be split into four parts, by summing over the intervals $[1,\lfloor m_I/2\rfloor]$, $[\lfloor m_I/2\rfloor+1,m_I-1]$, $[m_I,\lfloor 3m_I/2\rfloor]$, and $[\lfloor 3m_I/2\rfloor +1,2m_I]$. It can be shown that the first and third sums do not exceed $\sum_{k=1}^{\lfloor m_I/2\rfloor}\log^2\left\lvert2\sin(\pi k/m_I) \right\rvert,$ and that the second and fourth sums do not exceed $\log^2(2)+\sum_{k=1}^{\lfloor m_I/2\rfloor}\log^2\left\lvert2\sin(\pi k/m_I) \right\rvert.$ Using (\ref{eq:deAraujo11}) and Lemma \ref{lemma:sumAsymptotic}, we find that \begin{align*}
        \norm{\Log\left(\zeta_n^{-an_I/2}-\zeta_n^{an_I/2}\right)}^2&< \lceil n_I/4 \rceil \cdot \left(\frac{4\pi^2}{24}m_I+6\log^2{2}\right) \\
        &\leq\left(\frac{n_I}{4}+\frac{3}{4}\right)\left(\frac{4\pi^2}{24}m_I+6\log^2{2}\right) \\
        &=\frac{\pi^2}{24}n+\frac{9}{2}\log^2{2}+\frac{3\log^2{2}}{2}n_I+\frac{3\pi^2}{24}\frac{n}{n_I}.
    \end{align*} The inequality $\lceil n_I/4\rceil \leq \frac{n_I+3}{4}$ holds because $n_I$ must be an integer. Taking square roots yields the result.
\end{proof}

\begin{lemma} \label{lemma:theirLemma4}
    Suppose $n\geq 5$ and $n\not\equiv 2\pmod{4}$. For any integer $a$ coprime to $n$ such that $1<a<n/2$, we have $\norm{\Log{(\xi_a)}}<2\left[(2^s-2)\sqrt{\frac{5\pi^2}{48}n+\frac{15}{2}\log^2{2}}+\sqrt{\frac{\pi^2}{6}n+6\log^2{2}}\right]$, where $s$ is the number of distinct prime factors of $n$.
\end{lemma}
\begin{proof}
    We use the strategy from \cite[Lemma 4]{deAraujo2024}, but we substitute our bound from Lemma \ref{lemma:theirLemma3} and make some improvements. Equation 18 in \cite{deAraujo2024} is
    \begin{equation}\label{eq:deAraujo18}
        \norm{\Log(\xi_a)}\leq \sum_{I\in\Gamma}\norm{\Log\left(\zeta_n^{-an_I/2}-\zeta_n^{an_I/2}\right)}+\sum_{I\in\Gamma}\norm{\Log\left(\zeta_n^{-n_I/2}-\zeta_n^{n_I/2}\right)},
    \end{equation} which follows from Lemma \ref{lemma:Ramachandra}. In the bound from Lemma \ref{lemma:theirLemma3}, the right hand side of (\ref{eq:LogBound}) is greatest if $n_I$ is chosen to be 1 (because $\frac{3\pi^2}{24}>\frac{3\log^2{2}}{2}$). However, if $n_I$ is not 1, then it must be at least 2 and cannot exceed $n/2$ (recall that $n_I$ must be a proper divisor of $n$). Provided that $n\geq 4$, which it is, one may verify that the right hand side of (\ref{eq:LogBound}) is maximised if $n_I$ is chosen to be $2$. That is, \begin{align}\label{eq:logBoundNullSet}
        \norm{\Log\left(\zeta_n^{-an_I/2}-\zeta_n^{an_I/2}\right)}&<\sqrt{\frac{\pi^2}{6}n+6\log^2{2}}\text{ if }I=\varnothing,\\
        \label{eq:logBoundOthers}
        \norm{\Log\left(\zeta_n^{-an_I/2}-\zeta_n^{an_I/2}\right)}&<\sqrt{\frac{5\pi^2}{48}n+\frac{15}{2}\log^2{2}}\text{ for any other }I\in \Gamma.
    \end{align}

    There are $2^s-1$ elements in $\Gamma$. In each sum in (\ref{eq:deAraujo18}), we therefore get one term bounded as in (\ref{eq:logBoundNullSet}) and $2^s-2$ terms bounded as in (\ref{eq:logBoundOthers}).
    That is, $$\norm{\Log(\xi_a)}< 2\left[(2^s-2)\sqrt{\frac{5\pi^2}{48}n+\frac{15}{2}\log^2{2}}+\sqrt{\frac{\pi^2}{6}n+6\log^2{2}}\right].$$

\end{proof}

\begin{lemma} \label{lemma:boundOnPhi}
    If $n$ is a positive integer with $s$ distinct prime factors, then $$\varphi(n)\leq n\left(1-\frac{1}{\sqrt[s]{n}}\right)^s.$$
\end{lemma}
\begin{proof}
    Let $p_1,p_2,\ldots,p_s$ be the $s$ distinct prime factors of $n$, then $\prod_{i=1}^sp_i\leq n$. We make two applications of the arithmetic mean-geometric mean (AM-GM) inequality:
    \begin{align*}
        \sqrt[s]{\prod_{i=1}^{s}\left(1-\frac{1}{p_i}\right)}&\leq\frac{s-\sum_{i=1}^{s}(1/p_i)}{s} \text{ via AM-GM} \\
        &= 1-\frac{\sum_{i=1}^s(1/p_i)}{s}\leq 1-\sqrt[s]{\frac{1}{\prod_{i=1}^{s}p_i}} \text{ via AM-GM} \\
        &=1-\frac{1}{\sqrt[s]{\prod_{i=1}^{s}p_i}}\leq 1-\frac{1}{\sqrt[s]{n}}.
    \end{align*}
    Notice that the last inequality is sharp if $n$ is squarefree. It then follows that $$
    \varphi(n)=n\prod_{i=1}^{s}\left(1-\frac{1}{p_i}\right)\leq n\left(1-\frac{1}{\sqrt[s]{n}}\right)^s.$$
\end{proof}
The following statement follows from \cite[Lemma 4.3]{GMR2005}; a proof can be found in \cite{GMR2005}. 
\begin{lemma} \label{lemma:coveringRadius}
For an $r$-dimensional lattice $\Lambda$, we have
    $\mu(\Lambda)\leq \frac{\sqrt{r}}{2}\lambda_r(\Lambda)$.
\end{lemma}

We can now prove Theorem \ref{thm:betterTheorem1}.
\begin{proof}
    Again, this is proved in roughly the same way as \cite[Theorem 1]{deAraujo2024}. Use the basis $\{\Log(\xi_a)\mid a\in G,a\neq 1\}$ for $\Log(C)$ as in Lemma \ref{lemma:Ramachandra}, whose cardinality (i.e. the rank of $C$) is $\frac{\varphi(n)}{2}-1$. Then \begin{align*}\lambda_r(\Log{\,C})\leq \max_{a\in G, a\neq 1}\norm{\Log\,\xi_a}<2\left[(2^s-2)\sqrt{\frac{5\pi^2}{48}n+\frac{15}{2}\log^2{2}}+\sqrt{\frac{\pi^2}{6}n+6\log^2{2}}\right].\end{align*} Using Lemma \ref{lemma:coveringRadius} with $\Lambda=\Log\,C$, we have
    \begin{align*}\mu(\Log\,C)&< \frac{\sqrt{r}}{2}\times 2\left[(2^s-2)\sqrt{\frac{5\pi^2}{48}n+\frac{15}{2}\log^2{2}}+\sqrt{\frac{\pi^2}{6}n+6\log^2{2}}\right]\\
    &< \sqrt{\frac{n}{2}}\left(1-\frac{1}{\sqrt[s]{n}}\right)^{s/2}\left[(2^s-2)\sqrt{\frac{5\pi^2}{48}n+\frac{15}{2}\log^2{2}}+\sqrt{\frac{\pi^2}{6}n+6\log^2{2}}\right],\end{align*} due to Lemma \ref{lemma:boundOnPhi} and the fact that $\sqrt{r}<\sqrt{\varphi(n)/2}$.
    The result follows by noting that $\mu(\Log\,\mathcal{O}_{\mathbb{K}}^*)\leq \mu(\Log\, C)$, because $\Log\,C$ is a sublattice of $\Log\, \mathcal{O}_\mathbb{K}^*$.
\end{proof}

\section*{Acknowledgements}
The author thanks Tim Trudgian for suggesting the problem and giving feedback on the manuscript.
The author is also grateful to Robson Ricardo de Araujo for answering questions about \cite{deAraujo2024} and to L\'eo Ducas for helpful comments on the results in this paper. The author acknowledges the support of an Australian Government Research Training Program Scholarship.

\bibliographystyle{plain}
\bibliography{refs}

\begin{thebibliography}{1}

\bibitem{ATPS2021}
Fernando Azpeitia~Tellez, Christopher Powell, and Shahed Sharif.
\newblock Geometry of biquadratic and cyclic cubic log-unit lattices.
\newblock {\em J. Number Theory}, 228:276--293, 2021.

\bibitem{CDPR2016}
Ronald Cramer, L\'eo Ducas, Chris Peikert, and Oded Regev.
\newblock Recovering short generators of principal ideals in cyclotomic rings.
\newblock In {\em Advances in cryptology---{EUROCRYPT} 2016. {P}art {II}}, volume 9666 of {\em Lecture Notes in Comput. Sci.}, pages 559--585. Springer, Berlin, 2016.

\bibitem{deAraujo2024}
Robson~Ricardo de~Araujo.
\newblock An upper bound on the covering radius of the logarithmic lattice for cyclotomic number fields.
\newblock {\em Discrete Math.}, 347(1):Paper No. 113665, 6, 2024.

\bibitem{GMR2005}
Venkatesan Guruswami, Daniele Micciancio, and Oded Regev.
\newblock The complexity of the covering radius problem.
\newblock {\em Comput. Complexity}, 14(2):90--121, 2005.

\bibitem{Marcus1977}
Daniel~A. Marcus.
\newblock {\em {Number Fields}}.
\newblock Universitext. Springer-Verlag, New York-Heidelberg, 1977.

\bibitem{Washington1997}
Lawrence~C. Washington.
\newblock {\em {Introduction to Cyclotomic Fields}}, volume~83 of {\em Graduate Texts in Mathematics}.
\newblock Springer-Verlag, New York, second edition, 1997.

\end{thebibliography}

\end{document}